\documentclass[11pt]{amsart}

\usepackage{xcolor}

\title[Local structure of analytic stacks]{Local structure of the Teichm\"uller and the Riemann moduli stacks}

\author[A.K. Doan]{An Khuong Doan}
\address{Department of Mathematics, KU Leuven, Celestijnenlaan 200B, 3001 Leuven, Belgium.}\email{an-khuong.doan@kuleuven.be}
\address{CNRS-CMLS, UMR 7640, \'{E}cole polytechnique,  91128 Palaiseau Cedex, France.}


\date{August 26, 2024}

\usepackage{fullpage}

\usepackage{amssymb, amsmath, wasysym, mathrsfs, mathtools, color, stmaryrd,enumerate}
\usepackage[all, cmtip]{xy}
\usepackage{url}
\usepackage{tikz-cd}
\usetikzlibrary{fit,backgrounds}

\definecolor{hot}{RGB}{65,105,225}
\usepackage[pagebackref=true,colorlinks=true, linkcolor=hot ,  citecolor=hot, urlcolor=hot]{hyperref}

        \usepackage{eucal}
        \usepackage{stmaryrd} 
        \usepackage{mathrsfs} 
        \theoremstyle{plain}
        \newtheorem{theorem}{Theorem}[section]
        \newtheorem{corollary}[theorem]{Corollary}
        \newtheorem{lemma}[theorem]{Lemma}
        \newtheorem{question}[theorem]{Question}
          
        \newtheorem{proposition}[theorem]{Proposition}

        \theoremstyle{definition}
        \newtheorem{definition}[theorem]{Definition}
        \newtheorem{example}[theorem]{Example}
        
          \newtheorem{remark}[theorem]{Remark}
        \newtheorem*{example*}{Example}

        \theoremstyle{remark}
        \newtheorem*{remark*}{Remark}  

\numberwithin{equation}{section}

\newcommand{\stX}{\mathcal{X}} 
\newcommand{\stY}{\mathcal{Y}} 
\newcommand{\Sp}{\operatorname{Spec}}

\DeclareMathOperator{\Spec}{Spec}

\DeclareMathOperator{\Aut}{Aut}
\DeclareMathOperator{\isom}{\underline{\mathrm{Isom}}}

\begin{document}
\setcounter{tocdepth}{1}
\begin{abstract} The goal of this note is to introduce an interesting question proposed by D. Rydh on an analytic version of the local structure of Artin stacks saying that near points with linearly reductive stabilizer, Artin stacks are étale-locally quotient stacks. We give some supporting evidence by verifying it on two fundamental classes of classical analytic moduli spaces: the Teichm\"uller moduli space and the Riemann moduli space of integrable complex structures whose analytic stack versions have been constructed by a recent work of L. Meersseman. 
\end{abstract}
\maketitle
\tableofcontents
\section{Introduction}
The theory of algebraic stacks was initially introduced by A. Grothendieck (cf.  \cite{grothendieck-techniques})  in the early $1960$s and through many generalizations finally rediscovered by M. Artin in the famous paper \cite{artin-versal} in order to deal with interesting algebraic moduli problems but non-representable by usual objects: schemes, algebraic spaces, etc. The original aim is to keep track of automorphism groups on moduli spaces, which is somehow the main reason causing their bad behavior. Actually, there is a philosophy saying that near points with linearly reductive stabilizer (or equivalently with linearly reductive automorphism group), algebraic stacks should admit at least étale-locally tractable presentations as quotient stacks \cite{alper-quotient, alper-kresch}. It has been recently confirmed by the following groundbreaking result (cf. \cite{ahr2}).

\begin{theorem} \label{theorem1.1} Let $\mathcal{X}$ be a quasi-separated algebraic stack, locally of finite type over
an algebraically closed field $k$, with affine stabilizers. Let $x \in \mathcal{X}(k)$ be a point and
$H \subset G_x$ be a subgroup scheme of the stabilizer such that $H$ is linearly reductive and
$G_x/H$ is smooth (resp. étale). Then there exists an affine scheme $\Sp(A)$ with an
action of $H$, a $k$-point $w \in \Sp(A)$ fixed by $H$, and a smooth (resp. étale) morphism
$$f:([\Sp(A)/H],w) \rightarrow (\mathcal{X},x)$$
such that $BH\cong f^{-1}(BG_x)$;  in particular, $f$ induces the given inclusion $H \rightarrow G_x$
on stabilizer group schemes at $w$. In addition, if $\mathcal{X}$ has affine diagonal, then $f$ can be arranged to be affine.
\end{theorem}
The beauty of this result lies in that it permits to reduce the étale-local study of such algebraic stacks to the study of quotient stacks of the form $[\Spec{A}/G]$ with $G$ linearly reductive. Fortunately, the latter is greatly-comprehended for example by means of the geometric invariant theory.  This undoubtedly contributes enormously to  a deeper understanding of complicated algebraic moduli problems among which we can mention the recent work of E. Arbarello and G. Saccà, dealing with the local structure of Bridgeland moduli spaces (cf. \cite{AS23}).  

Although, having been existing in parallel to algebraic stacks since its creation (it is worth noticing that Grothendieck's foundational work  in \cite{grothendieck-techniques} was first motivated by the comprehension of the Teichmüller space of isomorphism classes of marked closed Riemann surfaces of a given genus - an analytic object), the literature on analytic stacks is surprisingly scarce and often appeared in an extremely general form following Grothendieck's formalism for example in \cite{bn}, \cite{toe-higher} and recently  \cite{Ino19},  \cite{meersseman-the}. 
Here, by analytic stacks, we mean those defined over the site of complex analytic spaces with smooth analytic atlases (cf. Definition \ref{definition2.6}) and the site is endowed with the usual Euclidean topology. Therefore, a deeper study of such a theory is of great demand. Once having achieved this, we shall start to be able to translate results from the algebraic side to the analytic side and vice versa. For example, an appropriate analytic counterpart of Theorem \ref{theorem1.1} is also expected to hold. In fact, through e-mail correspondence, D. Rydh suggested the following analog.

\begin{question}\label{conjecture1.2} If $\mathcal{X}$ is an analytic stack with appropriate assumptions. Then, near a $\mathbb{C}$-point with linearly reductive stabilizer, $\mathcal{X}$ is étale-locally equivalent to a quotient stack?
\end{question} In reality, there are many analytic moduli spaces whose analytic stacky version enjoys the mentioned property e.g. moduli space of polarized manifolds admitting a constant scalar curvature Kähler metric \cite{DN18}, moduli space of Fano manifolds with Kähler-Ricci solitions \cite{Ino19}, moduli space of semistable Higgs bundle on a closed Riemann surface \cite{Fa22}, moduli space of polystale holomorphic vector bundles on a compact complex Kahl\"{e}r manifold \cite{BS21}, in which the main ingredient of the proof is the construction of local Kuranishi slices with respect to some reductive group actions and then one tries to see if these local models could be patched together. Last but not least, the definition of the analytic moduli space of degree $d$ log Del Pezzo surfaces \cite[Definition 3.13  and Definition 3.14]{OSS16} is more or less given with the affirmation of Question \ref{conjecture1.2} as the main defining property. As a matter of fact,  if the answer to Question \ref{conjecture1.2} was positive, it would definitely help to better understand complicated analytic moduli problems at least locally. 

The objective of this note is to give some more supporting evidence to Question \ref{conjecture1.2}. In other words, apart from Deligne-Mumford analytic stack (Theorem \ref{dms}) and good quotient stacks (Theorem \ref{quost}), we test it in classical moduli problems: the Teichmüller  space and the Riemann moduli space of integrable complex structures. 
As might be expected, due to the analytic nature, a proof for the positivity of Question \ref{conjecture1.2} if existed could be different from the proof of Theorem \ref{theorem1.1} at least concerning tools used. In the sequel, this observation will be clear in the verification of the four aforementioned cases.  Namely, let $M$ be a  connected compact oriented differentiable manifold of even dimension. Let $\mathscr{T}(M)$( resp. $\mathscr{M}(M)$) be the Teichmüller stack (resp. Riemann moduli) stack of integrable complex structures on $M$, both defined over the site of complex analytic spaces (cf. Section \ref{s4.1}). Then according to a delicate work of L. Meersseman \cite{meersseman-the}, $\mathscr{T}(M)$ and $\mathscr{M}(M)$ are analytic stacks whose étale local structure near points with reasonable stabilizer can be further described by our main result (see Theorem \ref{theorem5.1} and Theorem \ref{theorem5.2} below).
\begin{theorem} \label{theorem3} The answer to Question \ref{conjecture1.2} is yes for $\mathscr{T}(M)$ and $\mathscr{M}(M)$.
\end{theorem}
The principal idea of our proof is to make use of our previous work on an equivariant version of Kuranishi theory on deformations of integrable complex structures (cf. \cite{doan-equivariant}). Namely, according to the foundation work of Kodaira-Spencer-Kuranishi \cite{kodaira-complex,kuranishi-deformation}, given any integrable complex structure, there exists an analytic space, often called the Kuranishi space $S$, parametrizing all the nearby complex structures. In addition, the associated family is semi-universal in the sense that any other family is defined by pullback by this family by a holomorphic map whose differential is unique. Furthermore,  if the stabilizer $G$ of the complex structure under consideration is linearly reductive, then there exists a $G$-action on the Kuranishi space compatible with the associated family (cf. \cite{doan-equivariant}).  The rest of proof is reduced to the construction of an étale morphism from the corresponding quotient stack $[S/G]$ to $\mathscr{T}(M)$ and $\mathscr{M}(M)$. This final step is done via a careful reconsideration of Meersseman's construction of the smooth analytic groupoids associated to $\mathscr{T}(M)$ and $\mathscr{M}(M)$. It turns out that the quotient stack $[S/G]$ is isomorphic to the so-called Kuranishi stack from which natural inclusions into $\mathscr{T}(M)$ and $\mathscr{M}(M)$ can be shown to be étale by direct inspection. Finally, it is really worth emphasizing that Meersseman's construction \cite{meersseman-the} is crucial in determining the local structures of the analytic stacks under consideration in particular and of any kinds of analytic stacks constructed in his way in general (cf. Section \ref{s5} for a discussion).
\begin{remark}
Note that the answer to Question \ref{conjecture1.2} is affirmative when the stack is defined over the site of differentiable and its stabilizers are compact (cf. \cite[Theorem 2.4]{Zun06}) or over the site of complex manifolds (resp. topological spaces) and its stabilizers are finite - Deligne-Mumford analytic stack (cf. \cite{Noo05}, resp. \cite[Proposition 3.5]{bn} and Theorem \ref{dms} below as well).\end{remark}

Let us now outline the organization of the note. Section \ref{s2} is a quick introduction on analytic stacks in which we also clarify the equivalence between two existing notions of analytic stacks. Section  \ref{quotstr} investigates local structures of two familiar classes of analytic stacks where the answer to Question \ref{conjecture1.2} is positive: Deligne-Mumford stacks and good quotient stacks. In Section \ref{s3.3} and \ref{s4.1}, we recall definitions of the Teichmüller, of  the Riemann and of the Kuranishi object as classical spaces and also as analytic stacks. Section  \ref{s5} is dedicated to a proof of the étale-local structure theorem and some of its immediate applications which seems to be non-trivial, at first sight. This note is ended by some remarks on Question \ref{conjecture1.2} in Section \ref{s4.4}.

\subsection*{Acknowledgements} The author would like to profoundly thank D. Rydh for introducing this interesting problem to him and for sharing enthusiastically his insight on it as well as L. Meersseman for many useful discussions, suggestions and continuous encouragement without which this note would not be accomplished. Moreover, during the revision of this paper, the latter also shared with the author the latest version of his paper \cite{meersseman-kuranishi} which helped to clarify some ambiguities. Finally, the authors are grateful to the anonymous referee whose dedicated work improved the quality of the paper. 

The work is supported by the grant G0B3123N of N. Budur from FWO and the Centre de Math\'{e}matiques Laurent Schwartz, Palaiseau, France.
\section{A glimpse on analytic stacks} \label{s2}

   
Let $(Ana)$ be the category of finite-dimensional $\mathbb{C}$-analytic spaces (including everywhere non-reduced analytic spaces as well). We put a Grothendieck topology on $(Ana)$ by declaring that for each object $U\in (Ana)$, $\mathrm{Cov}(U)$ is the set of collections $\lbrace U_i \rightarrow U \rbrace_{i\in I}$ of morphisms for which each morphism is a standard topological open covering of $U$ and the map 
$$\prod_{i\in I} U_i \rightarrow U$$ is surjective. This is often called the \emph{Euclidean} site.

\begin{definition} A category fibered in groupoids $p: F \rightarrow (Ana)$ is a stack if the following two conditions hold:
\begin{enumerate}
\item[(i)] For any $X \in (Ana)$ and object $x,y \in F(X)$, the presheaf $\mathrm{Isom}(x,y)$ on $(Ana/X)$ is a sheaf.
\item[(ii)] For any covering $\lbrace X_i \rightarrow X \rbrace $ of an object $X \in (Ana)$, any descent data with respect to this covering is effective.
\end{enumerate}
\end{definition}

\begin{remark} Any category fibered in groupoids $F$, we can associate a stack $F^a$ enjoying the expected universal property among morphims of categories fibered in groupoids from $F$ to stacks. This process is called \emph{stackification}.
\end{remark}

\begin{definition} A morphism of stacks $f: \stX \rightarrow \stY$ is \emph{representable} if for any analytic space $U$ and morphism $y:U \rightarrow \stY$, the fiber product $\stX \times_{\stY,y}U$ is an analytic space. If we further let $P$ be a property of morphisms of analytic spaces such as surjectiveness, flatness, smoothness, étaleness, then $f$ is said to have $P$ if for any analytic space $U$ and any morphism of analytic spaces $y:U \rightarrow \stY$, the morphism $\stX\times_{\stY,y}U \rightarrow U$ of analytic spaces has $P$.

\end{definition}

\begin{definition} \label{definition2.6}  A stack $\stX$ is an \emph{analytic stack} if the following hold:
\begin{enumerate}
\item[(i)] The diagonal morphism
$$\Delta: \stX \rightarrow \stX \times \stX$$ is representable.
\item[(ii)] There exists a smooth surjective morphism $\pi: X \rightarrow\stX$ with $X$ an analytic space.
A morphism of analytic stacks $f: \stX \rightarrow \stY$ is a morphism of stacks. 
\end{enumerate}
Finally, an analytic stack $\stX$ is said to be \emph{Deligne-Mumford} if there exists further an étale surjective morphism from an analytic space to $\stX$.
\end{definition}
\begin{remark} Recall that a \emph{smooth} morphism of analytic spaces locally looks like a projection. A morphism of analytic spaces is \emph{finite} if it is proper and quasi-finite (or equivalently, proper with finite fibers). A morphism of analytic spaces is \emph{étale} if it is smooth and quasi-finite (cf. \cite[Definition 3.20 and Definition 3.22]{Ans23}). The latter is equivalent to a local isomorphism (\cite[Lemma 3.23]{Ans23}). Therefore, complex-analytically, étale morphisms mean precisely local isomorphisms whose fibers are not necessarily finite.
\end{remark}
\begin{remark} As in the algebraic setting, in Definition \ref{definition2.6}, Condition $(\mathrm{i})$ implies in particular that the morphism $\pi: X \rightarrow \stX$ in Condition $(\mathrm{ii})$. Conversely, we can show that the representability of $\pi: X \rightarrow \stX$ implies that of the associated diagonal morphism. Thus, it makes sense to talk about the smooth, surjective morphism $\pi: X \rightarrow \stX$.
\end{remark}
\begin{remark}\label{groupoid} Given an analytic stack $\stX$ and a surjective smooth presentation $U\rightarrow \stX$, we can form the smooth analytic groupoid $U\times_{\stX}U \rightrightarrows  U$ with source and target maps coming from the projections of the fiber product.  As in the algebraic case, it can be proved without any difficulties that the stackification of the prestack associated to this groupoid is isomorphic to $\stX$.
\end{remark}

In \cite{meersseman-the}, for technical reasons, the author defined the notion of analytic stacks in the inverse direction. Namely,  a stack is analytic if it is isomorphic to the stackification of the prestack associated to some smooth analytic groupoid.  It turns out that this definition and Definition \ref{definition2.6} are also equivalent in the analytic setting.  First of all, it would be nice to recall the following useful criterion which follows immediately from Definition \ref{definition2.6} and an obvious cartesian diagram.
\begin{lemma}  \label{lemma2.8} Let $\stX $ be a stack. The diagonal map $\Delta: \stX \rightarrow \stX \times \stX$ is representable if and only if for every analytic space $U$ and two objects $u_1, u_2 \in \stX(U)$, the sheaf $\isom(u_1,u_2)$ on $(Ana/U)$ is an analytic space.  Here, $(Ana/U)$ means the category of analytic spaces over $U$. \end{lemma} 

Given a smooth analytic groupoid $(G_1\rightrightarrows G_0)$, we use the standard notion $[G_1/G_0]$ to denote the stackification of its associated prestack. 
\begin{proposition} \label{proposition2.9} Any stack isomorphic to the stackification of the prestack associated to some smooth analytic groupoid is analytic in the sense of Definition \ref{definition2.6}.
\end{proposition}
\begin{proof} The proof is standard. We recall here for completeness. Let $\stX$ be a stack. Suppose that $\stX=[G_0/G_1] $  for some smooth analytic groupoid $(G_1\rightrightarrows G_0)$.  

For the sake of Lemma \ref{lemma2.8}, it suffices to check that for any analytic space $U$ and any objects $u,v\in \stX(U)$, the sheaf $\isom(u,v)$ on $(Ana/U)$ is an analytic space. Indeed,  without loss of generality, we can suppose that there exists a sufficiently small covering of $\lbrace U_i \rightarrow U \rbrace_{i\in I}$ of $U$ such that the $\stX_{\mid U_i}$ is  the restriction of the prestack $G$ associated to $(G_1\rightrightarrows G_0)$ on $U_i$ by the universal property of the stackification. This gives rise to the following cartesian diagram
$$\xymatrix{
\isom(u,v)_{\mid U_i}  \ar[d] \ar[r] &G_1\ar[d]^-{(s,t)} \\
 U_i \ar[r]^-{u_{\mid U_i},v_{\mid U_i}}  			& G_0\times G_0 
}$$
It means exactly that locally $\isom(u,v)$ is a fiber product of analytic spaces, hence an analytic space itself. Therefore, globally it is also an analytic space. 

Finally, it is easy to check that a surjective smooth presentation of $\stX$ is given by the natural map $G_0 \rightarrow \stX=[G_0/G_1]$.
\end{proof}Due to this proposition, from now we can interchangeably use these two definitions of analytic stacks. Moreover, Definition \ref{definition2.6} allows us to freely make use of classical results of the general theory of stacks over a site and in particular, results on algebraic stacks whose adaptations to the analytic case are almost similar.
\section{Quotient structures of analytic stacks} \label{quotstr}
\subsection{Deligne-Mumford analytic stacks} As an opening example, we show that Conjecture \ref{conjecture1.2} is true for a particular class of analytic stacks with finite stabilizer.
\begin{theorem}\label{dms}
Let $\pi: X \rightarrow \stX$ be a Deligne-Mumford analytic stack whose diagonal morphism $\Delta: \stX \rightarrow \stX \times \stX$ is closed with finite fibers. For any point $x\in \stX$, there exists an étale morphism: $$[S/G_x]\rightarrow (\stX,x)$$ where $S$ is an analytic space and $G_x$ is the stabilizer of $x$.
\end{theorem}
\begin{proof}
The proof is similar to that of \cite[Proposition 14.9]{Noo05} where the site of topological spaces is replaced by that of analytic spaces. We give a proof here for the sake of completeness. Consider the associated groupoid $R:=X\times_{\stX}X \rightrightarrows  X$ where the source and target maps are denoted by $s$ and $t$, respectively (see Remark \ref{groupoid}). Then we have the following cartesian diagram
$$\xymatrix{
R  \ar[d] \ar[r]^-{(s,t)} &X\times X\ar[d]^-{(\pi,\pi)} \\
 \stX \ar[r]^-{\Delta}  			& \stX \times \stX 
}$$
Pick a geometric point $\tilde x \in X$ lying over $x$. Then the stabilizer $G_x$ of $x$ can be identified with $s^{-1}(\tilde x)\cap t^{-1}(\tilde x)$. The latter is a finite group in $R$ by assumption since the morphism $(s,t)$ is a pull-back of the diagonal morphism $\Delta$. Moreover, $s$ and $t$ are also pull-backs of $\pi$ being étale so that they are étale, or equivalently local isomorphisms as well. Hence, for each $g\in G_x$, we can choose a neighborhood $\tilde U_g$ of $g$ in $R$ such that the restriction of $s$ on $\tilde U_g$ is an isomorphism of analytic spaces onto its image $s(\tilde U_g)$ which is a neighborhood of $x$. Once again the finiteness of $G_x$ implies that $\tilde U_g$'s can be arranged to be disjoint from each other. Replace now $\tilde U_g$ by $\hat U_g:=\tilde U_g \cap s^{-1}(\cap_{g\in G_x}s(\tilde U_g))$. Denote $A=R\setminus \cap_{g\in G_x}\tilde U_g $. Since $(s,t)$ being pull-backs of $\Delta$ is closed, $(s,t)(A)$ is closed in $(s,t)(R)$ and $(\tilde x,\tilde x)$ is not included in $(s,t)(A)$ by construction. It means that there exists a neighborhood $\tilde U$ of $\tilde x $ in $X$ such that $\tilde U\times \tilde U$ containing $(\tilde x, \tilde x)$ is disjoint from $(s,t)(A)$ in $(s,t)(R)$. Set $W_g:=\hat{U}_g \cap (s,t)^{-1}(\tilde U \times \tilde U)$, $U:=\cap_{g\in G_x}t(W_h)$ and $U_g:=W_g\cap t^{-1}(U)$. By construction, $U$ is an invariant neighborhood of $\tilde x$ and the restriction of the groupoid $R \rightrightarrows  X$ on $U$ is exactly $\coprod_{g\in G_x}U_g \rightrightarrows  U$ where the structure maps induced by $s$ and $t$ give rise to isomorphisms between $U_g$ and $U$ for each $g \in G_x$. 

Finally, we have an well-defined analytic action of $G_x$ on $U$, given by
$$g.u=t(s^{-1}(u)\cap U_h), \;\forall g\in G_x,u\in U$$ so that the groupoid $\coprod_{g\in G_x}U_g \rightrightarrows  U$ is equivalent to the translation groupoid $G_x \times U \rightrightarrows  U$ (cf. \cite[Lemma 7.7]{Noo05}). Note the analyticity of all the maps involved is guaranteed by that of $s$ and $t$. This finishes the verification.
\end{proof}
As in the algebraic case, this result tells us that Deligne-Mumford analytic stacks of this type can be covered by quotient stacks arising from finite group actions on analytic spaces.
\begin{example}\label{dmex}The most famous example to illustrate the philosophy mentioned in the introduction in general and Theorem \ref{dms} in particular might be the case of elliptic curve structures on the product of two $1$-spheres $M=\mathbb{S}^1\times \mathbb{S}^1 $. These are parametrized by the upper half plane $\mathbb{H}$ which is also the Teichm\"{u}ller stack $\mathscr{T}(M)$. Moreover, taking into account the action of $\mathrm{SL}(2,\mathbb{Z})$ on $\mathbb{H}$, The Riemann moduli stack is given by the quotient stack $[\mathbb{H}/\mathrm{SL}(2,\mathbb{Z})]$. Their stabilizers are all finite and étale-locally they are quotient stacks. Hence, this gives an example of Deligne-Mumford analytic stacks.
\end{example}
\subsection{Analytic quotient stacks} In the algebraic setting, the famous Luna's étale slice theorem \cite{Lun73} shows essentially that quotient stacks formed by smooth linear algebraic group actions on affine schemes of finite type over an algebraically closed field are étale-locally quotient stacks at points whose stabilizer is linearly reductive (cf. \cite[Lemma 9]{alper-kresch}). Therefore, one hopes that an analytic analogue of this should hold if we restrict ourselves to Stein spaces being the analytic counterpart of affine schemes and reductive complex Lie groups. Luckily, the latter ones can be shown to carry a compatible structure of linear algebraic group in the sense that every holomorphic finite-dimensional representation of theirs is rational (see \cite{HM59}). With this setup, we still have an analytic version of the Luna's slice theorem \cite{Sno82} due to D.M. Snow, whose stacky version might be read as follows. 

\begin{theorem} \label{quost}
Let $S$ be a finite-dimensional Stein space on which a complex reductive Lie group $G$ acts holomorphically. Suppose further either one of the following two conditions holds: given a point $x$ with closed orbit $G\cdot x$.

i) $G$ is a reductive complex Lie group.

ii) or more generally, there exists an equivariant holomorphic map $h:X \rightarrow \mathbb{C}^n$ embedding $G\cdot x$ and is a local immersion at $x$, for some $n$.

Then, if the stabilizer $G_x$ of $x$ is a reductive subgroup of $G$, there exists a $G_x$-invariant locally closed Stein subspace $x\in U\subset S$ such that the induced map $[U/G_x]\rightarrow [S/G]$ is an étale morphism of analytic stacks.
\end{theorem}
\begin{remark} a) When Condition $(i)$ holds, the stabilizer $G_x$ in the above theorem is automatically a reductive subgroup of $G$ by \cite{Mat60}.

b) Condition $(i)$ in fact implies Condition $(ii)$ by \cite[Proposition 2.5]{Sno82}. We do not know whether the reductivity   assumption on $G$ in Condition $(i)$ could be weakened so that it still implies Condition $(ii)$.

c) Concerning Condition $(ii)$, in the algebraic setting, one even has the existence of a $G$-equivariant closed embedding of $X$ into a finite-dimensional representation of $G$ in the first place (cf. \cite[Proposition 1.9]{Bri09}). If such an embedding existed in the analytic setting, the problem would be reduced to the algebraic case.
\end{remark}
\subsection{\'{E}tale global quotient} 
\begin{corollary}\label{egq} Let $X\rightarrow \stX$ be an analytic stack satisfying the assumptions of either Theorem \ref{dms} or Theorem \ref{quost}. Assume further that $\stX$ is compact, i.e. $X$ is compact in the usual sense. Then there exists an étale surjective morphism $[S/G]\rightarrow \stX$ where $S$ is an analytic space and $G$ is a finite group (respectively, a reductive complex Lie group).
\end{corollary}
\begin{proof}
It follows immediately from Theorem \ref{dms}(Theorem \ref{quost}, respectively) together with the compactness of $\stX$ (cf. \cite[Theorem 4.21]{ahr2} and its proof).
\end{proof}
\begin{remark} More generally, with the same proof, one can even show that Corollary \ref{egq} holds true for any compact analytic stack $\stX$ with reductive stabilizers such that near any point $x \in \stX$, $\stX$ is étale-locally equivalent to a quotient stack $[S_x/G_x]$ where $S_x$ is an analytic space together with an action of the stabilizer $G_x$ of $x$.

\end{remark}

\section{Teichm\"uller, Riemann and Kuranishi moduli stacks}\label{lastsec}
\subsection{Teichm\"uller, Riemann and Kuranishi spaces}\label{s3.3}
For the rest of the note, we fix a connected compact oriented differentiable manifold $M$ of even dimension which admits an integrable complex structure. Let $TM$ be its tangent bundle. Then it is well-known that each integrable complex structure on $M$ corresponds to an integrable complex operator $J:  TM \rightarrow TM$ such that 
\begin{enumerate}
\item[(i)] $J^2 =\mathrm{Id}$,
 \item[(ii)]  $J$ is compatible with the orientation of $M$,
 \item[(iii)]  $[T^{1,0},T^{1,0} ] \subset T^{1,0}$ where $T^{1,0}=\lbrace v-iJv\mid v\in TM  \rbrace$. 
\end{enumerate} 
Denote the set of such integrable complex operators by $\mathscr{I}(M)$ and for each point $J \in \mathscr{I}(M)$, the associated complex manifold by $X_J$. The Teichm\"uller (resp. Riemann) space is defined to be the quotient $\mathscr{I}(M)/\mathrm{Diff}^{0}(M)$ (resp, $\mathscr{I}(M)/\mathrm{Diff}^{+}(M)$) where $\mathrm{Diff}^{0}(M)$ (resp, $\mathrm{Diff}^{+}(M)$) is the group of diffeomorphisms of $M$, isotropic to the identity (resp. of diffeomorphisms of $M$ preserving the orientation). 

Now, let $J_0$ be a point in $\mathscr{I}(M)$ and $V_0$ an open neighborhood of $J_0$ in $\mathscr{I}(M)$. We are interested in integrable complex structures sufficiently close to $J_0$, which due to the pioneering work of Kodaira-Spencer-Kuranishi can be entirely encoded as elements $\omega$ of the space $A^{0,1}(\Theta_{J_0})$ of $(0,1)$-forms with values in the holomorphic tangent bundle $\Theta_{J_0}$ of $X_{J_0}$, satisfying the Maurer-Cartan equation 
$$\overline{\partial}\omega + \frac{1}{2}[\omega,\omega]=0.$$
Imposing an hermitian metric on $\Theta_{J_0}$, gives rise to the existence of a formal adjoint $\overline{\partial}^*$ of $\overline{\partial}$. Let $U_0$ be a neighborhood of $0$ in $A^{0,1}(\Theta_{J_0})$. Then the Kuranishi space is defined to be the
$$K_0 =\lbrace \omega \in U_0 \mid \overline{\partial}\omega + \frac{1}{2}[\omega,\omega]=\overline{\partial}^*\omega =0 \rbrace.$$
We now state Kuranishi’s fundamental results in \cite{Kur65} and \cite{kuranishi-deformation} in the formulation of \cite{meersseman-the}, which is more convenient for our purpose.
\begin{theorem}\cite[Theorem 3.4]{meersseman-the}  \label{theorem3.1}Let  $L_0$ be a complementary closed complex vector space of $H^0(X_{J_0},\Theta_{J_0})$ in $A^{0,0}(\Theta_{J_0})$. Then if $V_0$ and $U_0 $ are small enough, there exists an open neighborhood $W$ of $0$ in $L_0$ such that 
\begin{enumerate}
\item[(i)] $K_0$ is a finite-dimensional analytic subspace of $V_0$,
\item[(ii)] There exists an analytic isomorphism 
$$\Phi_0: V_0 \rightarrow K_0\times W_0 $$ such that 
   \begin{enumerate}
   \item[(ii.a)] the inverse map is given by 
     \begin{align*}
K_0 \times W_0 \rightarrow V_0 \\ 
(J,v)\mapsto J\cdot e(v)
\end{align*} where $e$ is the exponential map from $W_0\subset A^{0,0}(\Theta_{J_0})$ to $\mathrm{Diff}^0(M)$ and the action of $\mathrm{Diff}^0(M)$ on $\mathscr{I}(M)$ is given by 
$J\cdot f:=df^{-1}\circ J\circ  df.$
  \item[(ii.b)] the composition $K_0 \hookrightarrow V_0 \overset{\Phi_0}{\rightarrow}  K_0\times W_0 \overset{\mathrm{pr}_1}{\rightarrow} K_0$ is the identity map.
   \end{enumerate}
 \end{enumerate}

\end{theorem} 
\begin{remark} \label{remark3.2} We can equip a complex structure on the product $M\times K_0$ to obtain a family often called the Kuranishi family $\mathscr{K}_0\rightarrow (K_0,0)$ of $X_{J_0}$. Furthermore, this family is semi-universal in the sense that any other deformation of $X_{J_0}$ is obtained by pullback of this family under a holomorphic map whose differential is unique. Finally, if the stabilizer of of $J_0$, or equivalently the automorphism group $\Aut(X_{J_0})$ of $X_{J_0}$ is linearly reductive then the family $\mathscr{K}_0 \rightarrow (K_0,0)$ can be made $\Aut(X_{J_0})$-equivariant, extending the original $\Aut(X_{J_0})$-action on $X_{J_0}$ (cf. \cite{doan-equivariant}). In particular, we obtain a compatible holomorphic  $\Aut(X_{J_0})$-action on the Kuranishi space $K_0$. It also gives another proof of \cite[Theorem 3.5]{meersseman-kuranishi}.
\end{remark}
Theorem \ref{theorem3.1} together with the following lemma plays a crucial role in the construction of the Teichm\"uller, Riemann and Kuranishi stacks in the sequel. Denote by $\Aut^0(X_{J_0})$ the connected component of the automorphism group $\Aut(X_{J_0})$  of $X_{J_0}$ and set 
$\Aut^{1}(X_{J_0}):= \Aut(X_{J_0}) \cap \mathrm{Diff}^{0}(M)$. 
\begin{lemma}\cite[Lemma 4.2]{meersseman-the}  \label{lemma3.3} There exist isomorphisms 
\begin{align*}
W_0 \times \mathrm{Aut}^{k}(X_{J_0})   \rightarrow \mathcal{D}_i  \\ 
(J,v)\mapsto J\cdot e(v)
\end{align*}  where $\mathcal{D}_k $ is an open neighborhood of $\mathrm{Aut}^{k}(X_{J_0}) $ in $\mathrm{Diff}^0(M)$ for $k=0,1$.
\end{lemma}

\subsection{Teichm\"uller, Riemann and Kuranishi moduli stacks}\label{s4.1}
In this section, we recall the definitions of the Techm\"uller, Riemann and Kuranishi moduli stacks. For more details, the interested reader is referred to the beautiful paper \cite{meersseman-the} of L . Meeersseman. 

Given $S \in (Ana)$, an $M$-\textit{deformation} over $S$ is a proper and smooth morphisms of analytic spaces $\mathcal{X} \rightarrow S$ whose fibers are all compact complex manifolds diffeomorphic to $M$. Observe differentiably that such a deformation is a bundle over S with fiber $M$ and structural group $\mathrm{Diff}^{+}(M)$. It is called \textit{reduced} if the structural group is $\mathrm{Diff}^{0}(M)$. A morphism of reduced $M$-deformations is a morphism of deformations preserving the $\mathrm{Diff}^{0}(M)$-bundle structure.

The Riemann moduli stack (resp. the Teichm\"uller moduli stack) is defined as the stack $\mathscr{M}(M) \rightarrow (Ana)$ (resp. $\mathscr{T}(M) \rightarrow (Ana)$) such that 
\begin{enumerate}
\item[(i)] for each $S\in (Ana)$,  $\mathscr{M}(M)(S)$ (resp. $\mathscr{T}(M)(S)$) is the groupoid of isomorphism classes of (resp. reduced) $M$-deformations over $S$,
\item[(ii)] for each morphism $f : S'\rightarrow  S$ in $(Ana)$, $\mathscr{M}(M)(f)$ (resp. $\mathscr{T}(M)(f)$) is the obvious pullback morphism.
\end{enumerate}
\begin{remark}
\begin{enumerate} \label{remark4.1}
\item[(i)] The points of $\mathscr{M}(M)$ (resp. $\mathscr{T}(M)$) are complex manifolds $X_J$ up to holomorphic biholomorphisms (resp. holomorphic biholomorphisms $C^{\infty}$-isotropic to the identity).
\item[(ii)] For a given point $J$ in $\mathscr{M}(M)$ (resp. $\mathscr{T}(M)$), the stabiliser $G_J$ of $J$ is nothing but $\Aut(X_J)$ (resp. $\Aut^1(X_J):= \Aut(X_{J}) \cap \mathrm{Diff}^{0}(M)$). 
\end{enumerate}

\end{remark}
More than being merely stacks, the implicit analyticity of $\mathscr{M}(M)$ and $\mathscr{T}(M)$ has been recently discovered by the exceptional work of L. Meersseman. The strategy is to construct a smooth analytic groupoid and then to prove that the corresponding analytic stack is exactly $\mathscr{M}(M)$ (resp. $\mathscr{T}(M)$). From now on, we suppose that the $h^0$ is bounded (see the paragraph below \cite[Remark 2.12]{meersseman-the} for the definition of $h^0$).
\begin{theorem}\cite[Theorem 2.13 and Theorem 2.14]{meersseman-the}  The Riemann moduli stack $\mathscr{M}(M)$ and the Teichm\"uller stack $\mathscr{T}(M)$ are analytic stacks.
\end{theorem}
\begin{remark} An easy application of Proposition \ref{proposition2.9} gives the representability of the diagonal morphism associated to $\mathscr{M}(M)$ (resp. $\mathscr{T}(M)$) which is not very clear at first glance. This clarifies the representability question posed in  \cite[Section 13]{meersseman-the}.
\end{remark}

At last, we turn to the construction of the Kuranishi stack. Let $\mathrm{Diff}^0(M,\mathscr{K}_0)$ be the set of $C^{\infty}$-diffeomorphisms from $M$ to a fiber of the Kuranishi family (see Remark \ref{remark3.2}).  More precisely, an element of $\mathrm{Diff}^0(M,\mathscr{K}_0)$ is a pair $(J, F)$ where $J \in K_0$ and $F$ is a diffeomorphism from $M$ to the complex manifold $X_J$.  It is further said to be $(V_0,\mathcal{D}_k)$-\emph{admissible} (for $k=0,1$) if there exists a finite sequence $(J_i,F_i)\in \mathrm{Diff}^0(M,\mathscr{K}_0)$  (for $0\leq i \leq p$) such that
\begin{enumerate}
\item[(i)] $J_0=J$ and each $J_i \in K_0$,
\item[(ii)] Each $F_i\in \mathscr{D}_0 $ and $F=F_0\circ \ldots \circ F_p$,
\item[(iii)] $J_{i+1}=J_i\cdot F_i$.
\end{enumerate}

Consider the set $$\mathcal{A}_k:= \lbrace  (J,F)\in \mathrm{Diff}^0(M,\mathscr{K}_0) \mid  (J,F) \text{ is }  (V_0,\mathcal{D}_k)\text{-admissible}  \rbrace$$  together with maps
$$s(J,F)=J, \; t(J,F)=J\cdot F, \;  m((J,F),(J\cdot F,F'))=(J,F\circ F'),\;  i(J,F)=(J\cdot F,F^{-1})$$
with respect to which it can be proved that $$\mathcal{A}_k\overset{s}{\underset{t}\rightrightarrows} K_0$$ is a smooth analytic groupoid (cf. \cite[Proposition 4.6]{meersseman-the}).

\begin{definition} \label{definition4.4}The Kuranishi stack $\mathscr{A}_k$ is the stackification of the prestack associated to the smooth analytic groupoid
$$\mathcal{A}_k\overset{s}{\underset{t}\rightrightarrows} K_0$$  over the site $(Ana)$ for $k=0,1$.
\end{definition}

\subsection{A Luna \'{e}tale slice theorem for  $\mathscr{T}(M)$ and $\mathscr{M}(M)$  }\label{s5}
In this section, our main task is to prove the following local structure result whose spirit is inspired essentially by Theorem \ref{theorem1.1}. 
Recall that a point $J_0 \in \mathscr{T}(M)$ is a Kahl\"er point if the associated complex manifold $X_{J_0}$ is a Kahl\"er manifold.
\begin{theorem} \label{theorem5.1} Let $J_0 \in \mathscr{T}(M)$ be a point such that 
its stabilizer  $G_{J_0}$ is linearly reductive. Then there exists an \'{e}tale reprentable morphism
$$f:([S_0/G_{J_0}],0) \rightarrow (\mathscr{T}(M),J_0)$$
where $(S_0,0)$ is a pointed analytic space on which $G_{J_0}$ acts locally holomorphically and fixes $0$. If $I_0$ is further a Kahl\"er point, then $f$ can be chosen to be finite.
\end{theorem}
\begin{proof}  Let $\pi: \mathscr{K}_0 \rightarrow (K_0,0)$ be the Kuranishi family of the complex compact manifold $X_{J_0}$ associated to $J_0$. Remark \ref{remark4.1} tells us that $G_{J_0}$ is nothing but $\Aut^1(X_{J_0})$. By assumption, $\Aut^1(X_{J_0})$ is linearly reductive so that there exist local $\Aut^1(X_{J_0})$-actions on $\mathscr{K}_0$ and on $K_0$ such that $\pi$ is $\Aut^1(X_{J_0})$-equivariant and the restriction of the $\Aut^1(X_{J_0})$-actions on $\mathscr{K}_0$ on the central fiber  $X_{J_0}$ is exactly the initial $\Aut^1(X_{J_0})$-actions on $X_{J_0}$ (cf. \cite{doan-equivariant}). In particular, the $\Aut^1(X_{J_0})$-actions on $K_0$ fixes $0$ which gives rise to a translation groupoid
$$\Aut^1(X_{J_0}) \times K_0 \rightrightarrows K_0.$$ 
Now, we claim that it is in fact equivalent to the smooth analytic groupoid $$\mathcal{A}_1\overset{s}{\underset{t}\rightrightarrows} K_0$$ in Definition \ref{definition4.4} . Indeed, take an element $(J,F) \in \mathcal{A}_1$. By definition, there exists a finite sequence $(J_i,F_i)\in \mathrm{Diff}^0(M,\mathscr{K}_0)$  (for $0\leq i \leq p$) such that
\begin{enumerate}
\item[(i)] $J_0=J$ and each $J_i \in K_0$,
\item[(ii)] Each $F_i\in \mathscr{D}_0 $ and $F=F_0\circ \ldots \circ F_p$,
\item[(iii)] $J_{i+1}=J_i\cdot F_i$.
\end{enumerate}
For the sake of Lemma \ref{lemma3.3}, the following decomposition is available
$$F_0=g_0\circ e(\xi_0)$$ where $g_0\in \Aut^1(X_{J_0})$ and $\xi_0\in W_0$. Replace $F_0$ by $F_0\circ e(-\xi_0)$ then $F_0  \in \Aut^1(X_{J_0}$. For each $i$ ($i=1,\ldots,p$), applying Lemma \ref{lemma3.3} once again we get $$e(\xi_{i-1})\circ F_i=g_i \circ e(\xi_i)$$  where $g_i \in \Aut^1(X_{J_0})$ and $\xi_i\in W_0$.  Replace $F_i$ by $e(\xi_{i-1})\circ F_i \circ e(-\xi_i)$ then $F_i \in \Aut^1(X_{J_0})$. In this way, we obtain a new $(V_0,\mathcal{D}_1)$-admissible pair $(J,F)$ such that every $F_i$ is in $\Aut^1(X_{J_0})$. In other words, it means precisely that $(\mathcal{A}_1\overset{s}{\underset{t}\rightrightarrows} K_0)$ is equivalent to $(\Aut^1(X_{J_0}) \times K_0 \rightrightarrows K_0)$. Consequently, the Kuranishi stack $\mathscr{A}_1$ is isomorphic to the quotient stack $[K_0/ \Aut^1(X_{J_0})]$ associated to the translation groupoid.

By the definition of the Teich\"uller moduli stack (cf. Section \ref{s4.1}), a neighborhood of $I_0$ in $\mathscr{T}(M)$ corresponds to reduced $M$-deformations where all the complex structures associated to the fibers are in a neighborhood $V$ of $I_0$ in $\mathscr{I}(M)$ (cf. Section \ref{s3.3} for the definition of $\mathscr{I}(M)$). Now, if we take $V$ to be open, connected and sufficiently small, an easy application of \cite[Theorem 6.2]{meersseman-kuranishi} provides an \'etale representable morphism $$f_1: ( \mathscr{A}_1,J_0) \rightarrow (\mathscr{T}(M),J_0), $$ or equivalently an étale representable morphism $$f:( [K_0/ \Aut^1(X_{J_0})],0) \rightarrow (\mathscr{T}(M),J_0). $$ This finishes the proof of the first statement. The second statement follows from the fact the morphism $f_1$ is a finite morphism when $J_0$ is a Kahl\"er point (cf. \cite[Theorem 6.2]{meersseman-kuranishi}).
\end{proof} 
\begin{theorem} \label{theorem5.2} Let $J_0 \in \mathscr{M}(M)$ be a point such that the connected component $G_{J_0}^0$ of its stabilizer  $G_{J_0}$ is linearly reductive. Then there exists an étale representable morphism
$$f:([S_0/G_{J_0}^0],0) \rightarrow (\mathscr{M}(M),J_0)$$
where $(S_0,0)$ is a pointed analytic space on which $G_{J_0}^0$ acts locally holomorphically and fixes $0$.
\end{theorem}
\begin{proof} On one hand, by the same argument as in the proof of Theorem \ref{theorem5.1}, we can prove that the Kuranishi stack $\mathscr{A}_0$ is nothing but the quotient stack $([K_0/ \Aut^0(X_{J_0})],J_0) $. Once again, due to \cite[Theorem 6.2]{meersseman-kuranishi}, we obtain an \'{e}tale representable morphism 
$$g_0: ( \mathscr{A}_0,J_0) \rightarrow (\mathscr{T}(M),J_0), $$ or equivalently an étale representable morphism $$g_0:( [K_0/ \Aut^0(X_{J_0})],J_0) \rightarrow (\mathscr{T}(M),J_0). $$ 

On the other hand, near $I_0$, the local groupoids for $\mathscr{T}(M)$ and $\mathscr{M}(M)$ can be taken to be the following groupoids
$$\mathcal{T}\overset{s}{\underset{t}\rightrightarrows} K_0 \text{ and } \mathcal{M}\overset{s}{\underset{t}\rightrightarrows} K_0\text{, respectively}$$
where
$$\mathcal{T}:= \lbrace  (J,F)\in \mathrm{Diff}^0(M,\mathscr{K}_0) \mid  J\cdot F \in K_0 \rbrace$$  and 
$$\mathcal{M}:= \lbrace  (J,F)\in \mathrm{Diff}^+(M,\mathscr{K}_0) \mid  J\cdot F \in K_0 \rbrace$$ and the source and the target are the obvious candidates. This gives rise to a natural inclusion $$h_0:(\mathscr{T}(M),J_0) \rightarrow \mathscr{M}(M).$$  By a direct adaptation of the proof of \cite[Theorem 6.2]{meersseman-kuranishi}, we also can prove $h_0$ is an \'{e}tale representable morphism of analytic stacks.    Indeed, for any morphism $u$ from an analytic space $B$ to $(\mathscr{T}(M),J_0)$, the morphism $f_1$ in the following fibered product $$\xymatrix{
B\times_u (\mathscr{T}(M),J_0)  \ar[d]^{f_1} \ar[r] & (\mathscr{T}(M),J_0) \ar[d]^-{h_0} \\
B \ar[r]^-{u}  			& \mathscr{M}(M) 
}$$
can be identified with the natural projection 
$$\pi_1: B\times \mathrm{Diff}^+(M)/\mathrm{Diff}^0(M) \rightarrow B $$ on the first factor. Note that the second factor is discrete. This implies in particular that the latter morphism is \'{e}tale and so is the former.

Composing this morphism with $g_0$, we get again an étale representable morphism 
$$f:( [K_0/ \Aut^0(X_{J_0})],0) \rightarrow (\mathscr{M}(M),J_0). $$ 
This ends the proof.
\end{proof}

Theorem \ref{theorem5.1} and Theorem \ref{theorem5.2} allow us to derive étale-locally many interesting properties of the Teichm\"uller stack and the Riemann moduli stack near points with linearly reductive stabilizer. These are apparently not quite obvious at first glance (see \cite[Section 1.3]{ahr2} for a comparison in the algebraic setting). For simplicity, throughout the rest of this section, by $\mathfrak{X}$, we mean either the Teichm\"uller stack or the Riemann moduli stack and by $G$, we mean either $\mathrm{Aut}^1(X_0)$ or $\mathrm{Aut}^0(X_0)$ (respectively) whenever $X_0$ is the reference point.

\begin{corollary} \label{coro6.1}Let $I_0$ be a $\mathbb{C}$-point of $\mathfrak{X}$ with linearly reductive stabilizer. Then,

(i) There is an étale neigborhood of $I_0$  with a closed embedding into a smooth analytic stack.

(ii) If $I_0$  is a non-exceptional K\"ahlerian point (cf. \cite[Definition 6.4]{meersseman-kuranishi}) then the Teichm\"uller stack $(\mathscr{T}(M),I_0)$ is isomorphic to a quotient stack.

(iii) Any representation $V$ of $G$ extends to a vector bundle over an étale neighborhood of $I_0$ .

(iv) In a neighborhood of $I_0$, the cotangent complex $\mathcal{L}_{\mathfrak{X}}^{\bullet}$ of $\mathfrak{X}$ is étale-locally quasi-isomorphic to a complex of two terms.


\end{corollary}

\begin{remark}
From combining $(iv)$ and $(ii)$ in Corollary $\ref{coro6.1}$, we deduce that near a non-exceptional Kahl\"erian point, the cotangent complex of the Teichm\"uller moduli stack is actually quasi-isomorphic to a $2$-term complex.
\end{remark}
Intuitively, everything should work in an apparently similar way for holomorphic vector bundles on a fixed complex compact manifold $X_0$. In fact, corresponding versions of Theorem \ref{theorem3.1} and Lemma \ref{lemma3.3} are also available for this case (cf. \cite{miy-kuranishi}). Therefore, by using the same strategy as Meersseman's one, one can show  that the Kuranishi stack, the Teichm\"{u}ller stack and the Riemann stack of holomorphic vector bundles on $X_0$ are analytic in the sense of Definition \ref{definition2.6}.
Now, on one hand, inspired by Meersseman's work, we can start talking about conditions under which the Kuranishi stack is isomorphic to the Teichm\"{u}ller stack. Of course, this time, conditions should be imposed on the holomorphic vector bundle under consideration. On the other hand, in a previous work \cite{doan-group}, we also proved that for any holomorphic vector bundle on which a linearly reductive group acts holomorphically, the existence of an equivariant structure on the Kuranishi family is guaranteed. Thus, by a similar argument, we can prove that locally around a point whose stabilizer is linearly reductive, the Teichm\"{u}ller stack and the Riemann moduli stack of holomorphic vector bundles are quotient stacks. This justifies the paragraph below \cite[Remark 4.2]{doan-group}. More generally,  for any object under deformation, whose a theory of Kuranishi's type is available (cf. \cite{MN}), similar results can be provided for the corresponding analytic stacks of such objects.

Next, we give some typical examples illustrating the main results.
\begin{example} The most simple example is the Deligne-Mumford analytic stack $[\mathbb{H}/\mathrm{SL}(2,\mathbb{Z})]$ introduced in Example \ref{dmex}. By definition, it is nothing but the Riemman moduli stack of marked complex structures on $M:=\mathbb{S}^1\times \mathbb{S}^1 $. Moreover, the stabilizers are all finite groups and then étale-locally they are quotient stacks by Theorem \ref{theorem5.2}. 
\end{example}
\begin{example} (Why reductivity)\label{example6.2}
Consider the product of two $2$-dimensional spheres $M: = \mathbb{S}^2\times \mathbb{S}^2$ and its Teichm\"{u}ler stack $\mathscr{T}(M)$. Note that $\mathbb{C}$-points of  $\mathscr{T}(M)$ can be identified with $\mathbb{Z}$. Each integer a corresponds to the Hirzebruch surface $\mathbb{F}_{2a}$ (cf. \cite[Example 12.6]{meersseman-the}). Take $I_0$ to be $1$ (or equivalently $\mathbb{F}_2$). According to Corollary \ref{coro6.1}(ii), locally around $I_0$,  $\mathscr{T}(M)$ is isomorphic to the Kuranishi stack $(\mathscr{A}_1,I_0)$ (cf. Definition \ref{definition4.4}). However, $\mathscr{A}_1$ is a quotient stack if and only if the action of $\mathrm{Aut}^1(\mathbb{F}_2)$ extends to a compatible  $\mathrm{Aut}^1(\mathbb{F}_2)$-action on the Kuranishi $S_0$ of $\mathbb{F}_2$. The latter is not true in general due the counter-example constructed in \cite{doan-a} (or \cite{Doan23}). So we can not expect an isomorphism in this situation. Moreover, what prevents the group extension is the fact that  $\mathrm{Aut}^1(\mathbb{F}_2)$ is $\mathbb{C}^3 \rtimes \mathrm{GL}(2,\mathbb{C})$ being clearly non-reductive. In brief, the reductivity assumption is somehow optimal.
\end{example}
\begin{example} With the same data as in Example \ref{example6.2}, we take $I_0$ now to be $0$ corresponding to the Kahl\"{e}r manifold $\mathbb{P}^1\times \mathbb{P}^1$. First, it is interesting to see that $0$ is actually an open dense subset in the Techm\"{u}ller space in the sense that any other point can be deformed onto $0$ by an arbitrary small deformation (cf. \cite[Example 2.16]{kodaira-complex}). Moreover, its Kuranishi space is reduced to that point and its automorphism group is connected and reductive.  Thus, a neighborhood of $0$ in the Techm\"{u}ller stack is nothing but the classifying stack of its automorphism group.\end{example}


Finally, we have encountered the situation that a complex structure with non-reductive automorphism group can be deformed to a new one with reductive automorphism group (the second Hirzebruch surface, for example). From what we have done so far, this is somehow an obstacle which prevents the moduli stack from locally admitting a good quotient presentation. Hence, it is natural to ask whether the converse can happen for any kind of geometric objects under deformation. In the case of complex structures, this question amounts to justifying whether there exists a compact complex manifold with reductive automorphism group, of which some deformation has a fiber with non-reductive automorphism group. The answer turns out to be yes if one considers a Hopf surface of type $\mathrm{IV}$ which can be deformed to another Hopf surface of type $\mathrm{II}_b$ (see \cite{Weh81} for more details).
\subsection{Comments on assumtionps of Question \ref{conjecture1.2}} \label{s4.4} To end the note, we record here comments on which kind of assumptions should be imposed to clarify a bit the word ``appropriate" in Question \ref{conjecture1.2}, observed by D. Rydh.  For algebraic groups, we have Chevalley's theorem saying that every smooth connected algebraic group is an extension of an abelian variety by a linear algebraic group (or equivalently, affine group) but this fails in the analytic situation (cf. \cite{Che60}). However, we also have the ``anti-affine" classification (cf. \cite{Bri09a} ) saying that every connected algebraic group is an extension by a linear algebraic group by an ``anti-affine group" (they are always smooth, connected and commutative
and have no non-constant regular functions). It seems that the same holds for
complex groups (cf. \cite{Mor65}): every connected complex Lie group is an extension of a Stein group by a connected commutative group without non-constant holomorphic functions. Theorem \ref{theorem1.1} is false without the assumption that all stabilizers are affine (cf. \cite[Example 1.5]{ahr2}). Similarly, it is expected that the analytic version would fail if not all
stabilizers are Stein.



\end{document}